\documentclass[12pt]{amsart}

\usepackage{amsmath}
\usepackage{amssymb}
\usepackage{amsthm}
\usepackage{graphicx}
\usepackage{amscd}
\usepackage{epic, eepic, color}
\usepackage{url}
 \newtheorem{theorem}{\sc Theorem}[section]
 \newtheorem{lemma}[theorem]{\sc Lemma}
 \newtheorem{cor}[theorem]{\sc Corollary}
 \newtheorem{proposition}[theorem]{\sc Proposition}
 \newtheorem{prop}[theorem]{\sc Proposition}

\newtheorem{remark}[theorem]{\sc Remark}
\newtheorem{defi}[theorem]{\sc Definition}
\newtheorem{result}[theorem]{\sc Result}
\newcommand{\Pe}{\mathcal{P}}
\newcommand{\El}{\mathcal{L}}

\newcommand{\PS}{\Pe_S}
\newcommand{\LS}{\El_S}
\newcommand{\PG}{\mathrm{PG}}
\newcommand{\AG}{\mathrm{AG}}
\newcommand{\GF}{\mathrm{GF}}
\newcommand{\ind}{\mathrm{ind}}
\newcommand{\cut}[1]{}
\newcommand{\mb}{\makebox(0,0)}
\newcommand{\cd}{3}
\newcommand{\cdd}{3.7}
\newcommand{\dll}{1.75}
\newcommand{\fontdef}{\SetFigFontNFSS{10}{10}{\rmdefault}{\mddefault}{\updefault}}
\newcommand{\fontsmall}{\SetFigFontNFSS{8}{8}{\rmdefault}{\mddefault}{\updefault}}
\newcounter{connum}
\newcommand{\cnum}{\stepcounter{connum}C\arabic{connum}}

 \makeatletter
 \@addtoreset{equation}{section}
 \makeatother

 \makeatletter
 \@addtoreset{table}{section}
 \makeatother

\begin{document}

\title{Resolving sets and semi-resolving sets in finite projective planes}
\author[H\'eger]{Tam\'as H\'eger}
\author[Tak\'ats]{Marcella Tak\'ats}
\thanks{Authors were supported by OTKA grant K 81310. The first author was also supported by ERC Grant 227701 DISCRETECONT. 
Appeared in Electronic J.~Comb.\ \textbf{19:4} \#P30 (2012). This version contains minor corrections regarding the list of smallest resolving sets, see the supplementary page 21 for details.}

\begin{abstract}
In a graph $\Gamma=(V,E)$ a vertex $v$ is \emph{resolved} by a vertex-set $S=\{v_1,\ldots,v_n\}$ if its (ordered) distance list with respect to $S$, $(d(v,v_1),\ldots,d(v,v_n))$, is unique. A set $A\subset V$ is resolved by $S$ if all its elements are resolved by $S$. $S$ is a \emph{resolving set} in $\Gamma$ if it resolves $V$. The \emph{metric dimension of $\Gamma$} is the size of the smallest resolving set in it. In a bipartite graph a \emph{semi-resolving set} is a set of vertices in one of the vertex classes that resolves the other class.

We show that the metric dimension of the incidence graph of a finite projective plane of order $q\geq 23$ is $4q-4$, and describe all resolving sets of that size. Let $\tau_2$ denote the size of the smallest double blocking set in PG$(2,q)$, the Desarguesian projective plane of order $q$. We prove that for a semi-resolving set $S$ in the incidence graph of PG$(2,q)$, $|S|\geq \min \{2q+q/4-3, \tau_2-2\}$ holds. In particular, if $q\geq9$ is a square, then the smallest semi-resolving set in $\PG(2,q)$ has size $2q+2\sqrt{q}$. As a corollary, we get that a blocking semioval in PG$(2,q)$, $q\geq4$, has at least $9q/4-3$ points.

\end{abstract}

\maketitle

\centerline{\today}

\section{Introduction}

For an overview of resolving sets and related topics we refer to the survey of Bailey and Cameron \cite{BC}. Regarding these, we follow the notations of \cite{BC, Bailey}.
Throughout the paper, $\Gamma=(V,E)$ denotes a simple connected graph with vertex-set $V$ and edge-set $E$. For $x,y\in V$, $d(x,y)$ denotes the distance of $x$ and $y$ (that is, the length of the shortest path connecting $x$ and $y$). $\Pi=(\Pe,\El)$ always denotes a finite projective plane with point-set $\Pe$ and line-set $\El$, and $q$ denotes the order of $\Pi$. Sometimes $\Pi_q$ refers to the projective plane of order $q$.

\begin{defi}\label{resdef1}
$S=\{s_1,\ldots,s_k\}\subset V$ is a \emph{resolving set} in $\Gamma=(V,E)$, if the ordered distance lists $(d(x,s_1),\ldots,d(x,s_k))$ are unique for all $x\in V$. The \emph{metric dimension of $\Gamma$}, denoted by $\mu(\Gamma)$, is the size of the smallest resolving set in it.
\end{defi}

Equivalently, $S$ is a resolving set in $\Gamma=(V,E)$ if and only if for all $x,y\in V$, there exists a point $z \in S$ such that $d(x,z)\neq d(y,z)$.
In other words, the vertices of $\Gamma$ can be distinguished by their distances from the elements of a resolving set. We say that a vertex $v$ is \emph{resolved} by $S$ if its distance list with respect to $S$ is unique. A set $A\subset V$ is resolved by $S$ if all its elements are resolved by $S$.
If the context allows, we omit the reference to $S$. Note that the distance list is ordered (with respect to an arbitrary fixed ordering of $S$), the (multi)set of distances is not sufficient.

Take a projective plane $\Pi=(\Pe,\El)$. The incidence graph $\Gamma(\Pi)$ of $\Pi$ is a bipartite graph with vertex classes $\Pe$ and $\El$, where $P\in\Pe$ and $\ell\in\El$ are adjacent in $\Gamma$ if and only if $P$ and $\ell$ are incident in $\Pi$.
By a resolving set or the metric dimension of $\Pi$ we mean that of its incidence graph.
In \cite{Bailey}, Bailey asked for the metric dimension of a finite projective plane of order $q$. In Section \ref{SectRes} we prove the following theorem using purely combinatorial tools.

\begin{theorem}
\label{4q-4}
The metric dimension of a projective plane of order $q\geq23$ is $4q-4$.
\end{theorem}

It follows that the highly symmetric incidence graph of a Desarguesian projective plane attains a relatively large dimension jump (for definitions and details see the end of Section \ref{SectRes}).

Section \ref{SectCons} is devoted to the description of all resolving sets of a projective plane $\Pi$ of size $4q-4$ ($q\geq 23$).

One may also try to construct a resolving set for $\Pi=(\Pe,\El)$ the following way: take a point-set $\Pe_S\subset\Pe$ that resolves $\El$, and take a line-set $\El_S\subset\El$ that resolves $\Pe$. Then $S=\Pe_S\cup\El_S$ is clearly a resolving set. 
Such a resolving set $S$ is called a \emph{split resolving set}, and $\Pe_S$ and $\El_S$ are called \emph{semi-resolving sets}. 
By $\mu^*(\Pi)$ we denote the size of the smallest split resolving set of $\Pi$ (see \cite{Bailey}).
As we will see in Section \ref{SectSemires}, semi-resolving sets are in tight connection with double blocking sets.

\begin{defi}
A set $B$ of points is a \emph{double blocking set} in a projective plane $\Pi$, if every line intersects $B$ in at least two points. $\tau_2=\tau_2(\Pi)$ denotes the size of the smallest double blocking set in $\Pi$.
\end{defi}

Let $\PG(2,q)$ denote the Desarguesian projective plane of order $q$. In Section \ref{SectSemires} we use the polynomial method and the Sz\H{o}nyi-Weiner Lemma to prove that if a semi-resolving set $S$ is small enough, then one can extend it into a double blocking set by adding at most two points to $S$. This yields the following results.

\begin{theorem}
\label{semirestheo}
Let $S$ be a semi-resolving set in $\PG(2,q)$, $q\geq3$. Then $|S|\geq\min\{2q+q/4-3,\tau_2(\PG(2,q))-2\}$. If $q\geq 121$ is a square prime power, then $|S|\geq 2q+2\sqrt{q}$.
\end{theorem}

\begin{cor}\label{semirescor}
Let $q\geq 3$. Then $\mu^*(\PG(2,q))\geq\min\{4q+q/2-6,2\tau_2(\PG(2,q))-4\}$. If $q\geq 121$ is a square prime power, then $\mu^*(\PG(2,q))=4q+4\sqrt{q}$.
\end{cor}

A \emph{Baer subplane} in $\Pi_q$, $q$ square, is a set of $q+\sqrt{q}+1$ points that intersects every line in either one or $\sqrt{q}+1$ points. It is well known that for a square prime power $q$, the point-set of $\PG(2,q)$ can be partitioned into $q-\sqrt{q}+1$ mutually disjoint Baer subplanes.

When we refer to duality, we simply mean that as the axioms of projective planes are symmetric in points and lines,
we may interchange the role of points and lines (e.g., consider a set of lines as a set of points), and if we have a result regarding points, then we have the same (dual) result regarding lines. A finite projective plane is not necessarily isomorphic to its dual, however, $\PG(2,q)$ is.

Finally, let us fix the notation and terminology we use regarding $\PG(2,q)$. Let $\GF(q)$ denote the finite field of $q$ elements. For the standard representation of $\PG(2,q)$ as homogeneous triplets, we refer to \cite{Hirschfeld}. The representatives of a point or a line are denoted by a triplet in round or square brackets, respectively. 
Let $\AG(2,q)$ denote the desarguesian affine plane of order $q$. If we consider $\AG(2,q)$ embedded into $\PG(2,q)$, then we call the line in $\PG(2,q)$ outside $\AG(2,q)$ \emph{the line at infinity}, denote it by $\ell_{\infty}$, and we call its points \emph{ideal points}. 
We choose the co-ordinate system so that $\ell_{\infty}=[0:0:1]$, and we denote the point $(1:m:0)\in\ell_{\infty}$ ($m\in\GF(q)$) by $(m)$, and $(0:1:0)\in\ell_{\infty}$ by $(\infty)$. 
A line $[m:-1:1]$ with affine equation $y=mx+b$ is said to have slope $m$. The common (ideal) point of vertical lines or lines of slope $m$ is $(\infty)$ and $(m)$, respectively. 
The points of $\ell_{\infty}$ are also called \emph{directions}, and, with the exception of $(\infty)$, they are identified naturally with the elements of $\GF(q)$ by $(m)\in\ell_{\infty}\leftrightarrow m\in\GF(q)$. 
A point $(x:y:1)\in\AG(2,q)$ is also denoted by $(x,y)$.
Considering a set $S$, a line $\ell$ is an $(\leq i)$-secant, an $(\geq i)$-secant, or an $i$-secant to $S$ if $\ell$ intersects $S$ in at most $i$, at least $i$, or exactly $i$ points, respectively.

\section{Resolving sets in finite projective planes}\label{SectRes}

Note that the distance of two distinct points (lines) is always two, while the distance of a point $P$ and a line $\ell$ is $1$ or $3$, depending on $P\in\ell$ or $P\notin\ell$, respectively. Note that the elements of a set $S$ are resolved by $S$ trivially, as there is a zero in their distance lists.

\textbf{Notation.} 
\begin{itemize}
\item{For two distinct points $P$ and $Q$, let $PQ$ denote the line joining $P$ and $Q$.}
\item{For a point $P$, let $[P]$ denote the set of lines through $P$. Similarly, for a line $\ell$, let $[\ell]$ denote the set of points on $\ell$. Note that we distinguish a line from the set of points it is incident with.}
\item{Once a subset $S$ of points and lines is fixed, the terms \emph{inner point} and \emph{inner line} refer to the elements of $S$, while \emph{outer points} and \emph{outer lines} refer to those not in $S$.}
\item{For a fixed subset $S$ of points and lines we say a line $\ell$ is skew or tangent to $S$ if $[\ell]$ contains zero or one point from $S$, respectively. Similarly, we say a point $P$ is not covered or $1$-covered by $S$ if $[P]$ contains zero or one line from $S$, respectively.}
\item{For a subset $S$ of points and lines, let $\Pe_S=S\cap\Pe$, $\El_S=S\cap\El$.}
\end{itemize}

\begin{lemma}\label{lemma}
Let $S=\Pe_S\cup\El_S$ be a set of vertices in the incidence graph of a finite projective plane. Then any line $\ell$ intersecting $\Pe_S$ in at least two points (that is, $|[\ell]\cap\Pe_S|\geq2$) is resolved by $S$.
Dually, if a point $P$ is covered by at least two lines of $\El_S$ (that is, $|[P]\cap\El_S|\geq2$), then $P$ is resolved by $S$.
\end{lemma}
\begin{proof}
Let $\ell$ be a line, $\{P,Q\}\subset[\ell]\cap\Pe_S$, $P\neq Q$. Then any line $e$ different from $\ell$ may contain at most one point of $\{P,Q\}$, hence e.g.~$P\notin [e]$, hence $d(P,\ell)=1\neq d(P,e)=3$. By duality, this holds for points as well.
\end{proof}

\begin{proposition}\label{resdef3}
$S=\Pe_S\cup\El_S$ is a resolving set in a finite projective plane if and only if the following properties hold for $S$:
\begin{itemize}
\item[P1]{There is at most one outer line skew to $\Pe_S$.}
\item[P1']{There is at most one outer point not covered by $\El_S$.}
\item[P2]{Through every inner point there is at most one outer line tangent to $\Pe_S$.}
\item[P2']{On every inner line there is at most one outer point that is $1$-covered by $\El_S$.}
\end{itemize}
\end{proposition}
\begin{proof}
By duality and Lemma \ref{lemma} it is enough to see that $S$ resolves lines not in $\El_S$ that are skew or tangent to $\Pe_S$. Property 1 (P1) assures that skew lines are resolved. Now take a tangent line $\ell\notin\El_S$. If there were another line $e$ with the same distance list as $\ell$'s (hence $e\notin\El_S$), both $e$ and $\ell$ would be tangents to $\Pe_S$ through the point $[\ell]\cap\Pe_S$, which is not possible by Property 2 (P2).
\end{proof}

We will usually refer to the above alternative definition, but sometimes it is useful to keep the following in mind.

\begin{proposition}\label{resdef4}
$S=\Pe_S\cup\El_S$ is a resolving set in a finite projective plane if and only if the following properties hold for $S$:
\begin{itemize}
\item[PA]{Through any point $P$ there is at most one outer line not blocked by $\Pe_S\setminus\{P\}$.}
\item[PA']{On any line $\ell$ there is at most one outer point not covered by $\El_S\setminus\{\ell\}$.}
\end{itemize}
\end{proposition}
\begin{proof}
In other words, the Property A claims that on a point $P\in\Pe_S$ there may be at most one tangent from $\El\setminus\El_S$, while on a point $P\notin\Pe_S$ there may be at most one skew line. As the intersection point of two skew lines would validate the latter one, Property A is equivalent to Properties 1 and 2 of Proposition \ref{resdef3}. Dually, the same holds for the Properties with commas.
\end{proof}

\begin{proposition}\label{const}
The metric dimension of a projective plane of order $q\geq 3$ is at most $4q-4$.
\end{proposition}
\begin{proof}
We give a construction refined from Bill Martin's one of size $4q-1$ (cited in \cite{Bailey}), see Figure \ref{fig1}. Let $P$, $Q$, and $R$ be three arbitrary points in general position. Let $\Pe_S=[PQ]\cup [PR]\setminus\{P,Q,R\}$, and let $\El_S=[P]\cup[R]\setminus{PQ, PR, RQ}$. We will see that $S=\Pe_S\cup\El_S$ is a resolving set by checking the criteria of Proposition \ref{resdef3}.\\
P1: The only outer line skew to $\Pe_S$ is $RQ$.\\
P1': The only outer point uncovered by $\El_S$ is $Q$.\\
P2: As $q\geq 3$, $|[PQ]\cap\Pe_S|=|[PR]\cap\Pe_S|=q-1\geq2$. On a point $A\in[PQ]\setminus{P,Q}$ the only tangent is $AR$ (which is in $\El_S$). On a point $A\in[PR]\setminus{P,R}$ the only tangent is $AQ$.\\
P2': As $q\geq3$, $|[P]\cap\El_S|=|[R]\cap\El_S|=q-1\geq2$. The only point of a line $\ell\in[R]\setminus\{PR,RQ\}$ not covered by $[P]$ is $[\ell]\cap[PQ]$ (which is in $\Pe_S$). The only uncovered point on a line $\ell\in[P]\setminus\{PQ,PR\}$ is $[\ell]\cap RQ$. 

Hence $S$ is a resolving set of size $|\Pe_S|+|\El_S|=2q-2+2q-2=4q-4$.
\end{proof}

Our aim is to show that the metric dimension of a projective plane of order $q\geq23$ is $4q-4$, and to describe all resolving sets of that size.

\textbf{A general assumption:} from now on we suppose that $S=\Pe_S\cup\El_S$ is a resolving set of size $\leq 4q-4$.

\begin{proposition}\label{trivibound}
$2q-5\leq |\Pe_S|\leq 2q+1$, $2q-5\leq |\El_S|\leq 2q+1$.
\end{proposition}
\begin{proof}
Let $t$ denote the number of tangents that are not in $\El_S$. By Property $2$, $t\leq|\Pe_S|$. Recall that there may be at most one skew line that is not in $\El_S$ (Property 1). Then double counting the pairs of $\{(P,\ell)\colon P\in\Pe_S, P\in[\ell], |[\ell]\cap\Pe_S|\geq2\}$ we get $2(q^2+q+1-1-t-|\El_S|)\leq |\Pe_S|(q+1)-t$, whence 
\[ q|\Pe_S|\geq 2(q^2+q-|\El_S|)-t-|\Pe_S|\geq 2(q^2+q-(|\El_S|+|\Pe_S|))\geq 2(q^2-3q+4),\]
thus $|\Pe_S|\geq 2q-6+8/q$, and as it is an integer, $|\Pe_S|\geq 2q-5$. Dually, $|\El_S|\geq 2q-5$ also holds. From $|\Pe_S|+|\El_S|\leq 4q-4$ the upper bounds follow.
\end{proof}

This immediately gives $|S|\geq 4q-10$. We remark that a somewhat more careful calculation shows $2q-4\leq|\Pe_S|\leq 2q$ and hence $|S|\geq 4q-8$, provided that $q\geq 11$, but we don't need to use it.

\begin{remark}\label{Fano}
The metric dimension of the Fano plane is five.
\end{remark}
\begin{proof}
Suppose $|\Pe_S|\leq2$. Using the notaions of the proof of Proposition \ref{trivibound}, we see $2(7-1-|\Pe_S|-|\El_S|)\leq |\Pe_S|$ (as there is at most one two-secant through any point). This yields $|\El_S|\geq 6-3|\Pe_S|/2$, whence $|\Pe_S|+|\El_S| \geq 6-|\Pe_S|/2\geq 5$. Figure \ref{fig1} shows a resolving set of size five in the Fano plane.
\end{proof}

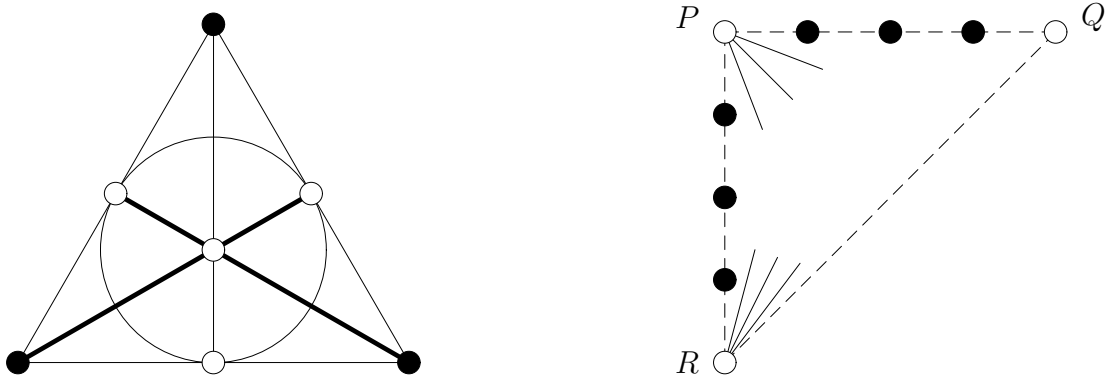
\begin{figure}
\begingroup\makeatletter\ifx\SetFigFontNFSS\undefined%
\gdef\SetFigFontNFSS#1#2#3#4#5{%
  \reset@font\fontsize{#1}{#2pt}%
  \fontfamily{#3}\fontseries{#4}\fontshape{#5}%
  \selectfont}%
\fi\endgroup%

\renewcommand{\mb}{\makebox(0,0)}
\renewcommand{\cd}{3}
\renewcommand{\dll}{2}

\setlength{\unitlength}{1mm}
\renewcommand{\dashlinestretch}{40}

\SetFigFontNFSS{12}{12}{\rmdefault}{\mddefault}{\itdefault}
\thinlines

\begin{picture}(160,48)(-80,0)


\filltype{white}
\put(-45,15){\circle*{30}}

\Thicklines
\drawline(-71,-0.15)(-32,22.35)
\drawline(-19,-0.15)(-58,22.35)
\drawline(-71,0.15)(-32,22.65)
\drawline(-19,0.15)(-58,22.65)

\thinlines
\drawline(-71,0)(-45,45)
\drawline(-71,0)(-19,0)
\drawline(-45,0)(-45,45)
\drawline(-19,0)(-45,45)

\filltype{white}
\put(-58,22.5){\circle*{\cd}}
\put(-32,22.5){\circle*{\cd}}
\put(-45,15){\circle*{\cd}}
\put(-45,0){\circle*{\cd}}

\filltype{black}
\put(-71,0){\circle*{\cd}}
\put(-19,0){\circle*{\cd}}
\put(-45,45){\circle*{\cd}}


\put(18,46){\mb{$P$}}
\put(72,46){\mb{$Q$}}
\put(18,0){\mb{$R$}}

\dashline{\dll}(23,0)(67,44)
\dashline{\dll}(23,0)(23,44)
\dashline{\dll}(23,44)(67,44)

\drawline(23,0)(27,15)
\drawline(23,0)(30,14)
\drawline(23,0)(33,13.2)

\drawline(23,44)(28,31)
\drawline(23,44)(32,35)
\drawline(23,44)(36,39)

\filltype{white}
\put(23,0){\circle*{\cd}}
\put(23,44){\circle*{\cd}}
\put(67,44){\circle*{\cd}}

\filltype{black}
\put(23,11){\circle*{\cd}}
\put(23,22){\circle*{\cd}}
\put(23,33){\circle*{\cd}}
\put(34,44){\circle*{\cd}}
\put(45,44){\circle*{\cd}}
\put(56,44){\circle*{\cd}}

\end{picture}
\caption{\label{fig1}On the left the black points and the thick lines form a resolving set of size five in the Fano plane. On the right the black points and the continuous lines form a resolving set of size $4q-4$.}
\end{figure}

\textbf{One more general assumption:} by duality we may assume that $|\Pe_S|\leq |\El_S|$. Thus, as $|\Pe_S|+|\El_S|\leq 4q-4$, $|\Pe_S|\leq2q-2$ follows.

\begin{proposition}\label{secantlength}
Let $q\geq 23$. Then any line intersects $\Pe_S$ in either $\leq 4$ or $\geq q-4$ points.
\end{proposition}
\begin{proof}
Suppose that $|[\ell]\cap\Pe_S|=x$, $2\leq x\leq q$. For a point $P\in[\ell]\setminus\Pe_S$, let $s(P)$ and $t(P)$ denote the number of skew or tangent lines to $\Pe_S$ through $P$, respectively; moreover, denote by $s$ the number of skew lines, and let $t$ denote the total number of tangents intersecting $\ell$ outside $\Pe_S$. Then counting the points of $\Pe_S$ on $\ell$ and the other lines through $P$ we get
$2q-2\geq |\Pe_S|\geq x + t(P) + 2(q-t(P)-s(P))$, equivalently, $x\leq 2s(P)+t(P)-2$. Adding up the inequalities for all $P\in[\ell]\setminus{\Pe_S}$ we obtain
\[(q+1-x)x\leq 2s+t-2(q+1-x).\]
Now, Proposition \ref{resdef3} yields that $s\leq |\El_S|+1$ and $s+t\leq (1+|\Pe_S|-x)+|\El_S|$ (here first we estimate the skew / tangent lines in question that are outside $\El_S$, then the rest), whence $2s+t\leq 2|\El_S|+|\Pe_S|-x+2$. Combined with the previous inequality we obtain
\[x^2-qx+4q-3\geq 0.\]
Assuming $q\geq 23$, the left hand side is negative for $x=5$ and $x=q-5$, therefore, as $x$ is an integer, we conclude that $x\leq 4$ or $x\geq q-4$.
\end{proof}

\begin{proposition}\label{standard}
Let $q\geq 23$. Then there exist two lines intersecting $\Pe_S$ in at least $q-4$ points.
\end{proposition}
\begin{proof}
By Proposition \ref{secantlength} every line is either a $\leq4$ or a $\geq(q-4)$-secant. Suppose to the contrary that every line intersects $\Pe_S$ in at most $4$ points except possibly one line $\ell$; let $x=|[\ell]\cap\Pe_S|\geq2$. Note that $x\leq 4$ is also possible. 
Let $n_i$ denote the number of $i$-secants to $\Pe_S$ different from $\ell$. To be convinient, let $n_0=s$ and $n_1=t$, and let $b=|\Pe_S|$. Then the standard equations yield
\begin{eqnarray*}
\sum_{i=2}^4{n_i} &=& q^2+q+1-s-t-1,\\
\sum_{i=2}^4{in_i} &=& (q+1)b-t-x,\\
\sum_{i=2}^4{i(i-1)n_i} &=& b(b-1)-x(x-1).
\end{eqnarray*}
Thus
\[0\leq\sum_{i=2}^{4}{(i-2)(4-i)n_i}=-\sum_{i=2}^{4}{i(i-1)n_i}+5\sum_{i=2}^{4}{in_i}-8\sum_{i=2}^{4}{n_i}=\]
\[-b^2+(5q+6)b+x(x-6)+3(s+t)+5s-8(q^2+q).\]
Substituting $s+t\leq |\Pe_S|+|\El_S|+1\leq 4q-3$, $s\leq|\El_S|+1\leq 2q+2$ and $x\leq q+1$, we get
\[0\leq -b^2+(5q+6)b-7q^2+10q-4.\]
By duality we assumed $b=|\Pe_S|\leq 2q-2$. For $b=2q-2$, the right hand side is $-q^2+20q-20$, which is negative whenever $q\geq 19$. Hence $b>2q-2$, a contradiction.
\end{proof}

Now we see that there exist two distinct lines $e$, $f$ such that $|[e]\cap(\Pe_S\setminus{[e]\cap [f]})|=q-l$ and $|[f]\cap(\Pe_S\setminus{[e]\cap [f]})|=q-k$ with $k\leq l\leq 5$.

Let $e\cap f=P$ and denote the set of points of $\Pe_S$ outside $e\cup f$ by $Z$.

\begin{proposition}\label{k+l}
Suppose $q\geq 23$. Then $k+l\leq 3$. Moreover, $l=3$ is not possible.
\end{proposition}
\begin{proof}

Then there are at least $q-1-|Z|$ skew or tangent lines through $P$ depending on $P\notin\Pe_S$ or $P\in\Pe_S$, respectively,
from which at most one may not be in $\El_S$, hence we found $\geq q-2-|Z|$ lines in $[P]\cap\El_S$. 
Among the $k(q-l)$ lines that connect one of the $k$ points in $[f]\setminus(\Pe_S\cup\{P\})$ with one of the $q-l$ points in $[e]\cap(\Pe_S\setminus\{P\})$ at most $k|Z|$ are not tangents to $\Pe_S$, but through a point in $|[e]\cap(\Pe_S\setminus{[e]\cap [f]})|=q-l$ only one tangent may not be in $\El_S$. 
Hence we find another $\geq k(q-l)-k|Z|-(q-l)=(k-1)(q-l)-k|Z|$ lines in $\El_S$. Interchanging the role of $e$ and $f$, we find yet another $\geq (l-1)(q-k)-l|Z|$ lines in $\El_S$. These three disjoint bunches give $|\El_S|\geq (k+l-1)q-(k+l+1)|Z|+(k+l)-2kl-2$.

Now as $(q-k)+(q-l)+|Z|\leq|\Pe_S|\leq 2q-2$, $|Z|\leq k+l-2$ holds, whence $|\El_S|\geq (k+l-1)q-(k^2+l^2+4kl-2(k+l))$.

We want to use that $2kl-2(k+l)\leq (k+l)^2-2(k+l)^{3/2}$, which is equivalent with $k^2+l^2\geq2(k+l)(\sqrt{k+l}-1)$. As $x\mapsto x^2$ is convex, $k^2+l^2\geq 2\left(\frac{k+l}{2}\right)^2=(k+l)^2/2\geq 2(k+l)(\sqrt{k+l}-1)$, where the last inequality follows from $(\sqrt{k+l}-2)^2\geq0$.

Therefore, $k^2+l^2+4kl-2(k+l)=(k+l)^2+2kl-2(k+l)\leq 2(k+l)^2-2(k+l)^{3/2}$, thus $2q+1\geq|\El_S|\geq (k+l-1)q-2(k+l)^2+2(k+l)^{3/2}$, hence 
\[\frac{2(k+l)^2-2(k+l)^{3/2}+1}{k+l-3}\geq q\geq23.\] 
The left hand side as a function of $k+l$ on the closed interval $[4,10]$ takes its maximum at $k+l=10$, and its value is $<20$. Hence $k+l<4$, i.e., $k+l\leq 3$.

Now suppose that $l=3$ (hence $k=0$). Recall that we may assume $|\Pe_S|\leq 2q-2$, hence $|Z|\leq1$. Then, as above, the number of lines $\in\El_S$ through $P$ and on the three points in $e\setminus{(\Pe_S\cup \{P\})}$ would be at least $q-3+2q-3=3q-6$, but $|\El_S|\leq 2q+1$, a contradiction.
\end{proof}

Right now we see that if $q\geq23$, then there are two lines containing at least $q-1$ and $q-2$ points (which also implies $|\Pe_S|\geq 2q-3$, hence $|\El_S|\leq 2q-1$ as well). Next we show this dually for lines. The dual arguments of the previous ones would also work, but we used duality to assume $|\Pe_S|\leq 2q-2$ to keep the technical bound on $q$ as low as possible, hence we make further considerations. Note that at most one point of $\Pe_S$ may not be covered by $e$ and $f$, whence Property A yields $|[P]\cap\El_S|\geq q-3$.

\begin{lemma}\label{FPL}
Let $e$ and $f$ be two distinct lines, $\{P\}=[e]\cap[f]$, $\{R_1,\ldots,R_q\}=[e]\setminus\{P\}$, $\{Q_1,\ldots,Q_q\}=[f]\setminus\{P\}$, $L\subset\El\setminus\{[P]\}$. Let $r_i=|L\cap[R_i]|$, $d_i=\max\{|L\cap[Q_i]|-1,0\}$, $m=d_1+\ldots+d_q$. Then for the number $C$ of points in $\Pe\setminus{([e]\cup[f])}$ covered by $L$, 
\[C\leq |L|(q-1)-r_i(|L|-r_i)+m\leq |L|q-r_i(|L|-r_i+1)\] holds, where $i\in\{1,\ldots,q\}$ is arbitrary.
\end{lemma}
\begin{proof}
Without loss of generality we may assume $i=1$. Let $d(R_j)$ ($2\leq j\leq q$) denote the number of lines in $L\cap[R_j]$ that intersect a line of $[R_1]\cap L$ on $f$. Then $\sum_{j=2}^q d(R_j)\leq m$ (count the lines in question through the points of $f$). Each line through $R_1$ covers $q-1$ points of $\Pe\setminus{([e]\cup[f])}$, while a line $h$ through $R_j$ ($2\leq j\leq q$) covers $q-1-r_1+\varepsilon$ new points, where $\varepsilon=1$ or $0$ depending on whether $h\cap f$ is covered by a line of $[R_1]\cap L$ or not, respectively. Therefore,
\[C\leq r_1(q-1)+\sum_{j=2}^q{\left(r_j(q-1-r_1)+d(R_j)\right)}=(q-1)\sum_{j=1}^q{r_j}-r_1\sum_{j=2}^{q}r_j+\sum_{j=2}^{q}{d(R_j)} \leq\]
\[|L|(q-1)-r_1(|L|-r_1)+m.\]
The second inequality follows immediately from $m\leq |L|-r_1$.
\end{proof}

\begin{proposition}\label{r_1}
There exists a point $R\in e\setminus{P}$ such that $|([R]\setminus{e})\cap\El_S|\geq q-1$. Moreover, if $l=2$, then $R\notin\Pe_S$.
\end{proposition}
\begin{proof}
We use the notations of Lemma \ref{FPL} with $L=\El_S\setminus{[P]}$. Let $W=\Pe\setminus{[e]\cup[f]}$. Suppose to the contrary that $r_i<q-1$ for all $1\leq i\leq q$ $(\star)$ .

\textbf{Case 1: $l=2$, $k\in\{0,1\}$.} $|\El_S|\leq 2q-1$ and $|[P]\cap\El_S|\geq q-3$ implies $|L|\leq q+2$. Keeping in mind that there may be one (but no more) point in $\Pe_S\cap W$, Property A' for the lines of $[P]\setminus{\{e,f\}}$ implies that $L$ must cover at least $(q-2)(q-1)+(q-2)=q(q-2)$ points of $W$. Then by Lemma \ref{FPL}, $q(q-2)\leq |L|q-r_i(|L|-r_i+1)\leq (q+2)q-r_i(q+3-r_i)$ (as the right hand side of the first inequality is growing in $|L|$, since $r_i<q$), which is equivalent with $r_i(q+3-r_i)\leq4q$. As the left hand side takes it minimum on the interval $[5,q-2]$ in $r_i=5$ and $r_i=q-2$, substituting $r_i=5$ yields $q\leq 10$, which does not hold. Hence $r_i\leq 4$ or $r_i\geq q-1$, thus by our assumption $r_i\leq 4$ for all $1\leq i\leq q$. 
Recall that $l=2$. Let $R_1$ and $R_2$ be the corresponding two points on $e\setminus{\Pe_S}$. According to Property A and considering the same ideas as in the proof of Proposition \ref{k+l}, we see that at least $q-2$ lines of $[R_1]\cup[R_2]$ must be in $L$. Thus $q-2\leq r_1+r_2\leq 8$, a contradiction. Therefore, without loss of generality we may conclude that $r_1\geq q-1$.

\textbf{Case 2: $k=l=1$.} In this case $\Pe_S=([e]\setminus{P,R})\cup([f]\setminus{P,Q})$, and $|\Pe_S|=|\El_S|=2q-2$. Note that the role of the lines $e$ and $f$ may be interchanged as they have the same combinatorial properties, thus we may expand the assumption $(\star)$ to $f$ as well, and it is also suitable to find the point $R$ on $f$.
Now $W\cap\Pe_S=\emptyset$, thus Property A' yields that at least $(q-1)^2$ points of $W$ must be covered by $L$, moreover, Property A implies $|\left([P]\setminus{\{e,f\}}\right)\cap\El_S|\geq q-2$, whence $|L|\leq q$ follows.

\textit{Subcase 2.1: $[P]\setminus{(\El_S\cup\{e,f\})}=\{\ell\}\neq\emptyset$.} As there may be at most one skew line outside $\El_S$ (Property 1'), this implies that the skew line $RQ$ is in $\El_S$. Lemma \ref{FPL} yields $(q-1)^2\leq |L|q-r_i(|L|-r_i+1)\leq q^2-r_i(q+1-r_i)$, equivalently, $r_i(q+1-r_i)\leq2q-1$. As in Case 1, this shows that $3\leq r_i\leq q-2$ is not possible, hence by $(\star)$ $r_i\leq 2$ for all $1\leq i\leq q$. Interchanging $e$ and $f$, we see that on any point in $f\setminus{P}$ there are at most two lines from $L$, hence $m\leq q/2$.
Then again by Lemma \ref{FPL}, $(q-1)^2\leq |L|(q-1)-r_i(|L|-r_i)+m\leq |L|(q-1)-r_i(|L|-r_i)+q/2$, equivalently, $r_i(q-r_i)\leq 3q/2-1$, hence $r_i\leq 1$ follows ($1\leq i\leq q$). Again, this holds for $f$ as well; that is, every point on $(e\cup f)\setminus{\{P\}}$ is covered at most once by $L$. The line $RQ$ is in $\El_S$, but then the points $R$ and $Q$ violate Property $2'$.

\textit{Subcase 2.2: $[P]\setminus{\{e,f\}}\subset \El_S$.} Then $|[P]\cap\El_S|\geq q-1$, thus $|L|\leq q-1$. Let $i\in\{1,\ldots,q\}$. Recall that $m\leq |L|-r_i$. Combined with Lemma \ref{FPL} we get $(q-1)^2\leq (q-1)^2-r_i(q-1-r_i)+m$, therefore $r_i(q-1-r_i)\leq m\leq q-1-r_i$, hence $r_i\leq 1$. As this is valid for the points of $f$ as well, $m=0$ follows. But then (under the assumption $(\star)$) $r_i=0$ would hold for all $1\leq i\leq q$, which is impossible.
\end{proof}

We have seen that $|\left([P]\setminus{\{e,f\}}\right)\cap\El_S|\geq q-3$. Now we prove that equality can not hold.

\begin{proposition}\label{p}
$|\left([P]\setminus{\{e,f\}}\right)\cap\El_S|\geq q-2$.
\end{proposition}
\begin{proof}
Suppose to the contrary that there exist two distinct lines, $g$ and $h$, such that $\{g,h\}\subset [P]\setminus{\{e,f\}}$, $\{g,h\}\cap\El_S=\emptyset$. Property A yields that (at least) one of them is blocked by a point $Z\in(\Pe_S\setminus{\{[e]\cup[f]\}})$. Thus $k=1$, $l=2$, $P\notin\Pe_S$ and $|\Pe_S\setminus{\{[e]\cup[f]\}}|=1$. Let $R$ be the point on $e\setminus{\{P\}}$ found in Proposition \ref{r_1}. Then $|\El_S\setminus{([P]\cup[R])}|\leq 1$, let $\ell$ denote this (possibly not existing) line. 
Take a line $r$ of $[R]\setminus{\{e\}}$ that does not go through any of the points $g\cap\ell$, $h\cap\ell$, and $Z$. Such a line exists as $q-3>0$. The points $r\cap g$ and $r\cap h$ show that $r$ violates Property A', a contradiction.
\end{proof}

\begin{proposition}\label{}
If $|\Pe_S|=2q-3$, then $|\El_S|\geq 2q-1$.
\end{proposition}
\begin{proof}
$|\Pe_S|=2q-3$ means that $k+l=3$. Let $R$ be the point on $e\setminus{\{P\}}$ found in Proposition \ref{r_1}, and denote by $R'$ the point $e\setminus \{\Pe_S\} \cup \{P,R\}$. We count the lines in $S$:
\begin{itemize}
\item By Proposition \ref{p} $|\left([P]\setminus{\{e,f\}}\right)\cap\El_S|\geq q-2$;
\item By Property 2 in Proposition \ref{resdef3}, through any point $F \in f\setminus{\{P,Q\}}$ at least one of the lines $FR, FR'$ has to be in $\El_S$ as both are tangents to $\Pe_S$ (at least $q-1$ lines);
\item By Property 1 in Proposition \ref{resdef3}, at least two of the three skew lines ($\ell_0, RQ, R'Q$) has to be in $\El_S$.
\end{itemize}
Altogether there are at least $2q-1$ lines in $\El_S$.
\end{proof}

\newcommand{\cO}{\mathcal{O}} \newcommand{\cOd}{\mathcal{O}^D}%

Thus, due to the assumption $|\Pe_S|\leq|\El_S|$, either $|\Pe_S|=2q-3$ and $|\El_S|\geq 2q-1$ or $2q-2\leq|\Pe_S|\leq|\El_S|$. This completes the proof of Theorem \ref{4q-4}.

Some lower bound on $q$ is necessary in Theorem \ref{4q-4}. As we have seen, the theorem fails for $q=2$ (Remark \ref{Fano}), since $\mu(\PG(2,2))=5$.
By Proposition \ref{const} we have $\mu(\PG(2,q))\leq 4q-4$ for $q\geq 3$. We have checked that $\mu(\PG(2,q))=4q-4$ for $q=3$ as well; however, the upper bound is not always tight. For $q=4$, a computer search showed $\mu(\PG(2,4))=10$, and $\PG(2,5)\leq 15$. 
We show a nice construction of size ten in $\PG(2,4)$. For basic facts about hyperovals see \cite{Hirschfeld}.
A hyperoval $\cO$ in $\PG(2,4)$ has six points, no tangents, $6\cdot5/2=15$ secants and six skew lines. Through any point $P\notin\cO$ there pass at most two skew lines, otherwise counting the points of $\cO$ on the lines through $P$ we obtained $|\cO|\leq 4$. Thus the set $\cOd$ of skew lines form a dual hyperoval. 
Now let $P\in\cO$ and $\ell\in\cOd$ be arbitrary, and let $\Pe_S=\cO\setminus\{P\}$, $\El_S=\cOd\setminus\{\ell\}$, $S=\Pe_S\cup\El_S$. Clearly, $\ell$ is the only skew line to $\Pe_S$, and there is precisely one tangent line on every point $R\in\Pe_S$ (namely $PR$). Thus P1 and P2 hold. Dually, P1' and P2' also hold, thus $S$ is a resolving set of size ten.

We remark that projective planes show an interesting example of highly symmetric graphs with large dimension jump. We follow the notations of \cite{BC}.
A vertex-set $B$ in a graph $\Gamma$ is called a \emph{base}, if the only automorphism of $\Gamma$ that fixes $B$ pointwise is the identity. The size of the smallest base of $\Gamma$ is called the \emph{base size of $\Gamma$}, and it is denoted by $b(\Gamma)$. 
As a resolving set of $\Gamma$ is a base, $b(\Gamma)\leq\mu(\Gamma)$ always holds. A repeatedly investigated question asks how large the gap $\delta(\Gamma)=\mu(\Gamma)-b(\Gamma)$ may be between these two parameters, referred to as the \emph{dimension jump of $\Gamma$} (see \cite{BC} and the references therein). 
Let $\Gamma$ be the incidence graph of $\PG(2,q)$. Then $\Gamma$ has order $n=2(q^2+q+1)$. It is well known that (the automorphism group of) $\Gamma$ is distance-transitive (that is, any pair $(u,v)$ of vertices can be transferred into any other pair $(u',v')$ of vertices by an automorphism of $\Gamma$ unless $d(u,v)\neq d(u',v')$.)
It is easy to see that $b(\Gamma)\leq5$ (four points are enough to fix the linear part of the collineation, and one more point forces the field automorphism to be the identity). Thus $\delta(\Gamma)\geq 4q-9$, which is quite large in terms of the order of $\Gamma$, roughly $2\sqrt{2n}$.

\section{Constructions}\label{SectCons}

Now we describe all resolving sets of size $4q-4$. Observing the nice symmetry and self-duality of the shown construction in Proposition \ref{const}, one might think that it is the only construction. However, this could not be further from the truth. In our somehow arbitrarily chosen system, there are 32 different constructions. Recall that we assume $|\Pe_S| \leq |\El_S|$.

By Propositions \ref{k+l}, \ref{r_1} and \ref{p}, we know that any resolving set $S=\Pe_S \cup \El_S$ of size $4q-4$ must contain the following structure $S^*=\Pe^*_S\cup\El^*_S$ of size $4q-6$ (see Figure \ref{S*}): two lines, $e$, $f$, where $[e] \cap [f] = \{P\}$, such that $|\Pe^*_S \cap([e] \setminus \{P\})|=q-2$, $|\Pe^*_S \cap([f] \setminus \{P\})|=q-1$, $|\El^*_S \cap([P] \setminus \{e,f\})|=q-2$, and for one of the points in $[e] \setminus (\Pe^*_S \cup \{P\})$, denote it by $R$, $|\El^*_S\cap([R]\setminus\{e\})|= q-1$. We denote by $R'$ the other point in $[e] \setminus (\Pe^*_S \cup \{P,R\})$, and let $\{Q\}= [f] \setminus (\Pe^*_S \cup \{P\})$, $\{\ell_0\}= [P] \setminus (\El^*_S \cup \{e, f\})$, $\{\ell_1\}= [R] \setminus (\El^*_S \cup \{e\})$. If $Q \notin \ell_1$, then let $T= f \cap \ell_1$.

\begin{figure}[!h]
\begin{center}
\begingroup\makeatletter\ifx\SetFigFontNFSS\undefined%
\gdef\SetFigFontNFSS#1#2#3#4#5{%
  \reset@font\fontsize{#1}{#2pt}%
  \fontfamily{#3}\fontseries{#4}\fontshape{#5}%
  \selectfont}%
\fi\endgroup%

\renewcommand{\mb}{\makebox(0,0)}
\renewcommand{\cd}{3}
\renewcommand{\dll}{2}

\setlength{\unitlength}{1mm}
\renewcommand{\dashlinestretch}{40}

\SetFigFontNFSS{12}{12}{\rmdefault}{\mddefault}{\itdefault}
\thinlines

\begin{picture}(50,50)(-10,-40)


\put(-2,3){\mb{$P$}}			
\put(-5,-27){\mb{$R'$}}			
\put(-5,-36){\mb{$R$}}			
\put(40,2){\mb{$Q$}}			
\put(15,3){\mb{$f$}}			
\put(-3,-15){\mb{$e$}}			

\dashline{\dll}(0,0)	(36,0)		
\dashline{\dll}(0,0)	(0,-36)		
\dashline{\dll}(0,0)	(9,-9)		
\dashline{\dll}(0,-36)	(10,-20)	

\put(13,-9){\mb{$\ell_0$}}		
\put(12,-18){\mb{$\ell_1$}}		

\drawline(0,0)(6,-12) 		
\drawline(0,0)(12,-6) 		

\drawline(0,-36)(5,-18)		
\drawline(0,-36)(14,-24)	
\drawline(0,-36)(17,-28)	

\filltype{white}
\put(0,0){\circle*{\cd}} 	
\put(36,0){\circle*{\cd}}	
\put(0,-27){\circle*{\cd}}	
\put(0,-36){\circle*{\cd}}	

\filltype{black}
\put(9,0){\circle*{\cd}}	
\put(18,0){\circle*{\cd}}	
\put(27,0){\circle*{\cd}}	
\put(0,-9){\circle*{\cd}}	
\put(0,-18){\circle*{\cd}}	
\end{picture}
\end{center}
\caption{\label{S*}The structure $S^*$ of size $4q-6$ that is contained in any resolving set of size $4q-4$.}
\end{figure}
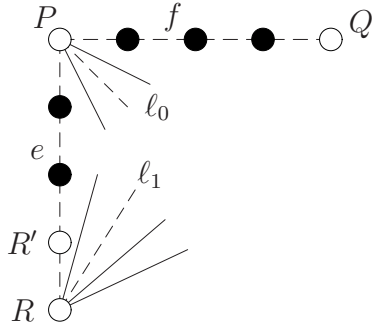
We have to complete this structure $S^*$ by adding two more objects to get a resolving set $S$. Assuming $|\Pe_S| \leq |\El_S|$, we have to add two lines or one line and one point, and then check the criteria of Proposition \ref{resdef3}. 

The problems of $S^*$ compared to the properties in Proposition \ref{resdef3} are the following:
\begin{itemize}
\item{P1 (outer skew lines to $\Pe^*_S$): $\ell_0$, $R'Q$, and if $Q\in \ell_1$, then $\ell_1$.}
\item{P1' (outer points not covered by $\El^*_S$):  $R'$, $\ell_0 \cap \ell_1$ and if $Q\in \ell_1$, then $Q$.}
\item{P2 (outer tangent lines through an inner point): if $Q \notin \ell_1$, then through $T=f \cap \ell_1$ the lines $\ell_1=RT$ and $R'T$ are tangents. Furthermore, if we add the intersection point of two outer skew lines (listed at P1), those will be two outer tangents through it.}
\item{P2' (outer $1$-covered points on an inner line): if $Q \notin \ell_1$, then on $RQ$ the points $Q$ and $\ell_0 \cap RQ$ are $1$-covered. Furthermore, if we add the line connecting two outer uncovered points (listed at P1'), those will be two outer one-covered points on it.}
\end{itemize}

These problems must be resolved after adding the two objects. By the letter ``C'' and a number we refer to the respective part of Figure \ref{consfig} at the end of the article. We distinguish the cases whether we add $\ell_1$ into $S$ or not, and whether $Q \in \ell_1$ or not.

{\bf I.} $\ell_1\in\El_S$ (see Figure \ref{consfig} (a)).

Now the problems of this construction compared to the properties are the following:

\begin{itemize}
\item{P1: $\ell_0$, $R'Q$.}
\item{P1': solved automatically, as the only outer not covered point is $R'$.}
\item{P2: if we add $\ell_0 \cap R'Q$, then $\ell_0$ and $R'Q$ are tangents through it.}
\item{P2': on $RQ$ the points $Q$ and $\ell_0 \cap RQ$ are $1$-covered.}
\end{itemize}

Note that in this case it does not matter whether $Q$ was on $\ell_1$ or not (see constructions C1-C4 on Figure \ref{consfig}). We can add one more line or one point to $S$ to solve these problems.

{\bf 1.a)} Adding $\ell_1$ and one more line:
in this case P2 is also solved automatically. To solve P1, we have to add one of the skew lines $\ell_0$ and $R'Q$. This automatically solves P2' as one of these lines covers one of the $1$-covered points on $RQ$ ($Q$, $\ell_0 \cap RQ$).
So we get two constructions: we can add $\ell_0$ and $\ell_1$ (\cnum), or $R'Q$ and $\ell_1$ (\cnum).

{\bf 1.b)} Adding $\ell_1$ and a point: 
because of P2, we cannot add the point $\ell_0 \cap R'Q$, so P2 is solved. To solve P1, we have to add a point on $\ell_0$ or $R'Q$. To solve P2', we have to add $Q$ (\cnum) or $\ell_0 \cap RQ$ (\cnum). Both choices will solve P1.

From now on we do not add $\ell_1$ to $\El_S$. We distinguish the cases whether $Q \in \ell_1$ or not.

\textbf{II.} $\ell_1\notin\El_S$, $Q\in\ell_1$ (see Figure \ref{consfig} (b)).

Now the problems are the following:

\begin{itemize}
\item{P1: $\ell_0$, $R'Q$ and $\ell_1$.}
\item{P1': $R'$, $\ell_0 \cap \ell_1$ and $Q$.}
\item{P2: we have to be careful if we add the intersection of two skew lines (listed at P1).}
\item{P2': we have to take care if we add the line joining two of the uncovered points (listed at P1').}
\end{itemize}

{\bf 2.a)} Adding two lines:
to solve P1, we have to add the skew lines $\ell_0$ and $R'Q$. But then we cannot solve P2', because $Q$ and $R'$ are two $1$-covered points on $R'Q$. So there are no such constructions.

{\bf 2.b)} Adding a point and a line:
to solve P1', we have to add or cover at least two of the points $Q$, $R'$ and $\ell_0 \cap \ell_1$. We cannot do this only by covering two points with a line, because then we cannot solve P2'. So we have to add one of these points.
\begin{enumerate}
\item Adding the point $Q$: P1 is solved automatically, as the only outer skew line is $\ell_0$. To solve P2, we have to add $R'Q$, since $R'Q$ and $\ell_1=RQ$ are outer tangent lines through $Q$. This solves P1' by covering $R'$. P2' is solved, as the only outer point on $R'Q$ is $R'$.
\item Adding the point $R'$: to solve P1, we have to add $\ell_0$. This solves P1', as $\ell_0$ covers $\ell_0 \cap \ell_1$. P2 and P2' are solved automatically. (This construction was the original example in the proof of Proposition \ref{const}.)
\item Adding the point $\ell_0 \cap \ell_1$: to solve P1', we have to cover $R'$ or $Q$. To solve P2, we have to add $\ell_0$, as $\ell_0$ and $\ell_1$ are outer tangent lines through the intersection point. But $\ell_0$ does not cover either $R'$ or $Q$, so there is no such a construction.
\end{enumerate}
So we get two constructions: we can add $Q$ and $R'Q$ (\cnum) or $R'$ and $\ell_0$ (\cnum). \newcounter{origconst}\setcounter{origconst}{\value{connum}}

{\bf Remark.} We already have two constructions such that $S$ contains $Q$. In fact if we add $Q$ to $\Pe_S$ these are the only possibilities. To solve P2, we have to add $\ell_1$ or $R'T$, and these are the constructions in 1.b) and 2.b), respectively. So from now on we do not add $Q$ to $\Pe_S$, and suppose that $Q \notin \ell_1$.

\textbf{III.} $\ell_1\notin\El_S$, $Q\notin\ell_1$,  $Q\notin\Pe_S$ (see Figure \ref{consfig} (c)).

The problems of this construction compared to the properties are the following:

\begin{itemize}\label{problemsIII}
\item{P1: $\ell_0$ and $R'Q$.}
\item{P1': $R'$ and $\ell_0 \cap \ell_1$.}
\item{P2: (i) through $T=f \cap \ell_1$ the lines $\ell_1=RT$ and $R'T$ are tangents; (ii) furthermore, we have to be careful if we add $\ell_0\cap R'Q$.}
\item{P2': (i) on $RQ$ the points $Q$ and $\ell_0 \cap RQ$ are $1$-covered; (ii) furthermore, we have to take care if we add the line joining $R'$ and $\ell_0\cap\ell_1$.}
\end{itemize}

{\bf 3.} Adding two lines ($\ell_1 \notin \El_S$, $Q \notin \Pe_S$, $Q \notin \ell_1$):
to solve P1, we have to add $\ell_0$ or $R'Q$. This automatically solves P1' by covering $\ell_0 \cap \ell_1$ or $R'$; and also solves P2'(i) by covering $\ell_0 \cap RQ$ or $Q$. To solve P2, we have to add $R'T$. P2'(ii) could be a problem if we add $R'Q$ and $(\ell_0 \cap \ell_1) \in R'Q$, but adding $R'T$ solves it as well by covering $R'$.
So we get two new constructions: we can add $\ell_0$ and $R'T$ (\cnum) or $R'Q$ and $R'T$ (\cnum).

From now on we have to add a point and a line to $S$.

{\bf Notation.} Let $U$ be an arbitrary point in $\{[e] \setminus \{P,R,R'\} \}$, and $V$ in $\{[f] \setminus \{P,Q,T\} \}$. Note that $U, V \in \Pe^*_S$.

Since most of the problems are caused by $R'$, we distinguish the cases whether we add $R'$ to $\Pe_S$ or not.

{\bf 4.} Adding $R'$ to $\Pe_S$ ($\ell_1 \notin \El_S$, $Q \notin \Pe_S$, $Q \notin \ell_1$):
P1 is solved as the only outer skew line to $\Pe_S$ is $\ell_0$. P1' is solved as the only outer point not covered by $\El_S$ is $\ell_0 \cap \ell_1$. P2 is solved as through $T$ the only outer tangent line is $\ell_1$. The new line cannot cause problem compared to P2'(ii) as $R'$ is an inner point now. The only problem we have to solve is P2'(i): on $RQ$ we have to cover $Q$ or $\ell_0 \cap RQ$ with a line.
\begin{enumerate}
\item $Q$: We can cover it by $f$ (\cnum), or $UQ$ or $R'Q$ (\cnum), both choices solve P2'.
\item $\ell_0 \cap RQ$: We can cover it by $\ell_0$ (\cnum) or the line connecting $U$ or $R'$ and $\ell_0 \cap RQ$ (\cnum), both choices solve P2'.
\end{enumerate}

From now on we suppose that $R'\notin \Pe_S$.

\textbf{IV.} $\ell_1\notin\El_S$, $Q\notin\ell_1$,  $Q\notin\Pe_S$, $R'\notin \Pe_S$ (see Figure \ref{consfig} (b)), we add one point and one line.

As we have to add a point and a line to $S$, we will go through systematically the possible addable lines keeping in mind the assumptions. First we check the line $e$, and then the lines which go through the points of $[e]$. We have to distinguish the points $P$, $U \in [e] \cap \Pe^*_S$ and $R'$ (as we have already seen the case adding $\ell_1$, the only outer line through $R$). We continue to refer to the problems listed in the case {\bf III.}
Note that P2'(ii) causes problem only in the last case (when adding a line through $R'$).

{\bf 5.} Adding $e$ to $\El_S$:
P1' is solved as $e$ covers $R'$. To solve P2'(i), we have to add $\ell_0 \cap RQ$ (as we do not add $Q$); this also solves P1. $\ell_0 \cap \ell_1 \notin RQ$ so this solves P2 only if $\ell_0 \cap RQ \in R'T$ (\cnum).

{\bf 6.} Adding a line through $P$:

{\bf 6. a)} Adding $f$:
P2'(i) is solved as $f$ covers $Q$. To solve P1', we have to add $\ell_0 \cap \ell_1$ (as we do not add $R'$); this also solves P1 and P2(i). By P2(ii), this works only if $\ell_0 \cap \ell_1 \notin R'Q$ (\cnum). 

{\bf 6. b)} Adding $\ell_0$:
P1 and P1' are solved as the only outer skew line and outer not covered point are $R'Q$ and $R'$. P2'(i) is solved as $\ell_0$ covers $\ell_0 \cap RQ$. To solve P2, we have to add an arbitrary point $Z$ on one of the lines $\ell_1$ (\cnum) and $R'T$ (\cnum). Note that in the former case we may add the point $R$ as well.

{\bf 7.} Adding a line through $U$:
we have to distinguish whether the added line meets $f$ in a point $V \in \{[f] \setminus \{P,Q\} \}$ or in $Q$. (Here we do not have to take care of $T$, the problems are the same for $UV$ and $UT$.)

{\bf 7. a)} Adding $UV$:
as $UV$ does not cover $R'$ and $Q$, and we do not add these points to $S$, to solve P1' and P2', one of the points $\ell_0 \cap \ell_1$ and $\ell_0 \cap RQ$ has to be covered by $UV$, and the other one has to be added to $S$. If $UV$ contains $\ell_0 \cap \ell_1$ and we add $\ell_0 \cap RQ$, then P1' and P2' are solved, as well as P1. This solves P2 only if $\ell_0 \cap RQ \in R'T$ (\cnum).
If $UV$ contains $\ell_0 \cap RQ$ and we add $\ell_0 \cap \ell_1$, then P1', P2', P1 and P2(i) are solved. By P2(ii) this works only if $\ell_0 \cap \ell_1 \notin R'Q$ (\cnum).

{\bf 7. b)} Adding $UQ$:
P2' is solved. If $\ell_0\cap\ell_1\notin UQ$, then we have to add $\ell_0\cap\ell_1$ to solve P1', which also solves P1 and P2(i). By P2(ii), this works only if $\ell_0 \cap \ell_1 \notin R'Q$ (\cnum).
If $\ell_0\cap\ell_1\in UQ$, then P1' is solved. To solve P1, we have to add a point on $\ell_0$ or on $R'Q$; to solve P2, we have to add point on $\ell_1$ or on $R'T$. By P2(ii), we cannot add $\ell_0 \cap R'Q$. Thus we may add $\ell_0\cap\ell_1$ (\cnum), $\ell_0 \cap R'T$ (\cnum) or $\ell_1 \cap R'Q$ (\cnum). Each choice will solve P1 and P2.

Now we check the cases when we add a line through $R'$. This solves P1' as $R'$ will be covered. Because of P2'(ii), we have to distinguish whether the added line contains $\ell_0 \cap \ell_1$ or not. 

{\bf 8.} Adding the line $g$ connecting $R'$ and $\ell_0 \cap \ell_1$:
to solve P2'(ii), we have to add $\ell_0 \cap \ell_1$. As $g$ cannot contain $\ell_0 \cap RQ$, it has to contain $Q$ to solve P2'(i). This also solves P1 and P2. So we get one construction: if $\ell_0 \cap \ell_1 \in R'Q$, we add $R'Q$ and $\ell_0 \cap \ell_1$ (\cnum).

{\bf 9.} Adding a line through $R'$ not containing $\ell_0 \cap \ell_1$:

We have to distinguish whether the added line meets $f$ in a point $V \in \{[f] \setminus \{P,T,Q\} \}$, in $T$ or in $Q$.

{\bf 9. a)} Adding $R'V$:
if $\ell_0 \cap RQ$ is not covered by $R'V$, we have to add it to $S$ in order to solve P2'(i). 
This solves P1, but solves P2 only if $\ell_0 \cap RQ \in R'T$ (\cnum). 
If $\ell_0 \cap RQ \in R'V$, then it solves P2'(i). To solve P1, we have to add a point on $\ell_0$ or $R'Q$; to solve P2, we have to add a point on $\ell_1$ or $R'T$, but we cannot add $\ell_0 \cap R'Q$ because of P2(ii). Adding $\ell_0 \cap R'T$ solves P1 and P2 without any further conditions (\cnum). Adding $\ell_0 \cap \ell_1$ (\cnum) or $\ell_1 \cap R'Q$ (\cnum) solves P1, but by P2(ii), it works only if $\ell_0 \cap \ell_1 \notin R'Q$.

{\bf 9. b)} Adding $R'T$:
P2(i) is solved. As $\ell_0 \cap \ell_1 \notin R'T$, P2'(ii) does not cause a problem.
If $\ell_0 \cap RQ \notin R'T$, we have to add it to $S$ to solve P2'(i). This solves P1 as well (\cnum). 
If $\ell_0 \cap RQ \in R'T$, then P2'(i) is solved. To solve P1, we have to add an arbitrary point $Z$ on $\ell_0$ (\cnum) or on $R'Q$ (\cnum) except the point $\ell_0 \cap R'Q$ (in both cases). Note that we may add the points $P\in\ell_0$ or $Q\in R'Q$ as well.

{\bf 9. c)} Adding $R'Q$: recall that $\ell_0 \cap \ell_1 \notin R'Q$. 
P1, P2' and P2(ii) are solved. To solve P2(i), we may add an arbitrary point $Z$ on $\ell_1$ (\cnum) or on $R'T$ (\cnum). Note that we may also add the point $R$ on $\ell_1$.

These are the all possibilities to get a resolving set of size $4q-4$ assuming $|\Pe_S| \leq |\El_S|$. There are four constructions with $|\Pe_S|>|\El_S|$, the duals of (C1), (C2), (C7) and (C8).

\section{Semi-resolving sets in $\PG(2,q)$} \label{SectSemires}

We recall the definition of a semi-resolving set for a projective plane. By duality, it is enough to discuss the case when a point-set resolves the lines of the plane.

\begin{defi}\label{semidef0}
Let $\Pi=(\Pe,\El)$ be a projective plane. $S=\{P_1,\ldots,P_n\}\subset\Pe$ is a \emph{semi-resolving set} if the ordered distance list $(d(\ell,P_1),\ldots,d(\ell,P_n))$ is unique for every line $\ell\in\El$.
\end{defi}

As in case of resolving sets, we can give two quick rephrasals of the above definition.

\begin{prop}\label{semidef1}
$S\subset\Pe_S$ is a semi-resolving set of $\Pi$ if and only if the following hold:
\begin{enumerate}
 \item{there is at most one skew line to $S$;}
 \item{through every point of $S$ there is at most one tangent line to $S$.}
\end{enumerate}
\end{prop}
\begin{proof}
Straightforward.
\end{proof}

Recall that $\tau_2(\Pi)$ denotes the size of the smallest double blocking set in the plane $\Pi$. 

\begin{result}[\cite{BB}]\label{BB}
Let $q\geq 9$. Then $\tau_2(\PG(2,q))\geq2(q+\sqrt{q}+1)$, and equality holds if and only if $q$ is a square.
\end{result}

Let $\mu_S(\Pi)$ denote the size of the smallest semi-resolving set in $\Pi$. The first part of the next proposition was pointed out by Bailey \cite{Bailey}.

\begin{proposition}\label{semicon}
\mbox{}
\begin{itemize}
\item[(i)]{$\mu_S\leq \tau_2-1$.}
\item[(ii)]{If there is a double blocking set of size $\tau_2$ that is the union of two disjoint blocking sets, then $\mu_S\leq\tau_2-2$.}
\item[(iii)]{In particular, if $q$ is a square prime power, then $\mu_S(\PG(2,q))\leq 2q+2\sqrt{q}$.}
\end{itemize}
\end{proposition}
\begin{proof}
Let $B$ a double blocking set, and let $P\in B$. Then $B\setminus\{P\}$ is clearly a semi-resolving set \cite{Bailey} (without a skew line). This proves (i).\\
Let $B$ be a double blocking set of size $|B|=\tau_2$ that is the union of two disjoint blocking sets $B_1$ and $B_2$, and let $S=(B_1\cup B_2)\setminus\{P_1,P_2\}$ for some $P_1\in B_1$, $P_2\in B_2$. 
We check the requirements of Proposition \ref{semidef1}. Clearly, there can be at most one skew line, namely $P_1P_2$.
Take a point $Q$, say, from $B_1\setminus\{P_1\}$. As $B_2$ intersects every line through $Q$, the only possible tangent line to $S$ through $Q$ is $QP_2$. Thus (ii) is proven.
If $q$ is a square, it is well-known that one can find two disjoint Baer subplanes in $\PG(2,q)$, so (iii) follows from (ii).\\
\end{proof}

\begin{proposition}
\label{2q-1}
Let $S$ be a semi-resolving set of $\Pi_q$. Then $|S|\geq 2q-1$.
\end{proposition}
\begin{proof}
Suppose to the contrary that $|S|\leq 2q-2$. Take a line $\ell$ tangent to $S$. Count the other tangents of $S$ through the points of $[\ell]$. As there are at most $2q-2$ tangents to $S$, there are at least two points in $[\ell] \setminus S$ with at most one tangent through them (besides $\ell$). At least one of them, denote it by $P$, is not contained in the (possible) skew line. At least $q-1$ lines through $P$ are at least $2$-secants to $S$, one line is at least a $1$-secant and $\ell$ is a tangent, so $|S|\geq 2q$, contradiction.

If there is no tangent line to $S$, take a point $R$ outside $S$. There is at most one skew line to $S$ through $R$, the other $q$ lines through $R$ intersect $S$ in at least two points, so $|S|\geq 2q$, contradiction.
\end{proof}

From now on we work in $\PG(2,q)$, and $S$ denotes a semi-resolving set in $\PG(2,q)$ of size $|S|=2q+\beta$, $\beta\in\mathbb{Z}$, $\beta\geq -1$. Almost every line intersects $S$ in at least two points: at most $|S|+1$ lines can be exceptional (that is, a $(\leq1)$-secant). It would be natural to note how many exceptional lines are on a point $P$, yet we need a less straightforward number assigned to the points.

\begin{defi}
Let $S$ be a semi-resolving set. For a point $P$, let the \emph{$i$th index} of $P$, denoted by $\ind_i(P)$, be the number of $i$-secants to $S$ through $P$. Let the \emph{index} of $P$, denoted by $\ind(P)$, be $2\ind_0(P)+\ind_1(P)$. For the sake of simplicity, denote the index of the ideal point $(m)$ by $\ind(m)$ instead of $\ind((m))$.
\end{defi}

Note that if $P\notin S$, then $\ind_0(P)\leq 1$ (as there is at most one skew line through $P$); if $P\in S$, then $\ind(P)\leq1$ (as there are no skew lines and at most one tangent through $P$).

We will use the following algebraic result. For $r\in\mathbb{R}$, let $r^+:=\max\{0,r\}$.

\begin{result}[Sz\H{o}nyi-Weiner Lemma \cite{Sziklai, WSzT}]\label{SzW} 
Let $u,v\in\GF(q)[X,Y]$. Suppose that the term $X^{\deg(u)}$ has non-zero coefficient in $u(X,Y)$ (that is, the degree of $u$ remains unchanged after substituting an element into the variable $Y$). For $y\in\GF(q)$, let $k_y:=\deg\gcd\left(u(X,y),v(X,y)\right)$, where $\gcd$ denotes the greatest common divisor of the two polynomials in $\GF(q)[X]$. Then for any $y\in\GF(q)$,
\[\sum_{y'\in\GF(q)}{\left(k_{y'}-k_y\right)^+} \leq (\deg{u}-k_y)(\deg{v}-k_y).\]
\end{result}

\begin{proposition}\label{quadineq}
Let $P\in\Pe\setminus{S}$. Assume $\ind(P)\leq q-2$, and $\beta\leq 2q-4$. Let $t$ be the number of tangents to $S$ plus twice the number of skew lines to $S$. Then
\begin{eqnarray}
\label{ineqt}\ind(P)^2-(q-\beta)\ind(P)+t\geq0,
\end{eqnarray}
and
\begin{eqnarray}
\label{ineqgen}\ind(P)^2-(q-\beta)\ind(P)+2q+\beta\geq0.
\end{eqnarray}
\end{proposition}
\begin{proof}
As $\ind_0(P)+\ind_1(P)\leq\ind(P)\leq q-2$, there are at least three lines through $P$ intersecting $S$ in at least two points, and all other lines intersect $S$ in at least one point except possibly the unique skew line. Among these three ($\geq2$)-secants, there must be one intersecting $S$ in $s\leq q-1$ points, otherwise $|S|\geq 3q + q-3=4q-3$ would hold, contradicting $\beta\leq 2q-4$.

Choose a coordinate system such that this $s$-secant line is the line at infinity $\ell_\infty$, $(\infty)\notin S$ and $P\neq(\infty)$. This can be done as $s\leq q-1$.
Let the set of the $|S|-s$ affine points of $S$ be $\{(x_i,y_i)\}_{i=1}^{|S|-s}$. Denote by $D$ the set of non-vertical directions that are outside $S$, $D\subset\GF(q)$. As $(\infty)\notin S$, $|D|=q-s$. 

Let \[R(B,M)=\prod_{i=1}^{|S|-s}(Mx_i+B-y_i)\in\GF(q)[B,M]\]
be the R\'edei polynomial of $S\cap\AG(2,q)$. If we substitute $M=m$ ($m\in\GF(q)$), then the multiplicity of the root $b$ of the one-variable polynomial $R(m,B)$ is the number of affine points of $S$ on the line $Y=mX+b$.
Fix $m\in\GF(q)$, and recall that $\ell_\infty$ is a $(\geq2)$-secant. Define $k_m$ as $k_m=\deg\gcd(R(m,B),(B^q-B)^2)$. Thus $k_m$ equals the number of single roots plus twice the number of roots of multiplicity at least two.

If $m\in D$, then the number of lines with slope $(m)$ that intersect $S\cap\AG(2,q)$ in at least one point or in at least two points is $q-\ind_0(m)$ and $q-(\ind_0(m)+\ind_1(m))$, respectively, thus $k_m=q-\ind_0(m)+(q-\ind_0(m)-\ind_1(m))=2q-\ind(m)$.

We use the Sz\H{o}nyi-Weiner Lemma with $u(B,M)=R(B,M)$ and $v(B,M)=(B^q-B)^2$. Note that the leading coefficient of both polynomials in $B$ is one, so Result \ref{SzW} applies. Let $P=(p)$ be our point on $\ell_\infty$ whose index shall be estimated. By the Lemma,

\[\sum_{m\in D}(k_{m}-k_p)\leq \sum_{m\in\GF(q)}(k_m-k_p)^+\leq (|S|-s-k_p)(2q-k_p)=(\ind(P)+\beta-s)\ind(P).\]

On the other hand, let $\delta=\sum_{m\in D}\ind(m)$; that is, we count the tangents and the skew line intersecting $\ell_\infty$ in $D$ with multiplicity one and two, respectively. Then $\sum_{m\in D}(k_m-k_p) = \sum_{m\in D}(\ind(P)-\ind(m)) = (q-s)\ind(P)-\delta$. Combined with the previous inequality we get
\begin{eqnarray}
\ind(P)^2-(q-\beta)\ind(P)+\delta\geq0.
\end{eqnarray}
As $\delta\leq t$, we obtain inequality \eqref{ineqt}.
Furthermore, as the (possibly not existing) skew line (counted with multiplicity two) may have a slope in $D$, and the (possibly not existing) tangents through the $s$ ($s\geq2$) points in $[\ell_\infty]\cap S$ are not counted in $\delta$, then $\delta\leq |S|-s+2\leq |S|=2q+\beta$. This gives inequality \eqref{ineqgen}.
\end{proof}

\begin{proposition}\label{indices}
Suppose $\beta\leq q/4-5/2$. Let $P\notin S$. Then $\ind(P)\leq 2$ or $\ind(P)\geq q-\beta-2$.
\end{proposition}
\begin{proof}
Suppose that $P\notin S$ and $\ind(P)\leq q-2$ (in order to use the inequality \eqref{ineqgen} in Proposition \ref{quadineq}). Substituting $\ind(P)=3$ or $\ind(P)=q-\beta-3$ into \eqref{ineqgen}, we get
$\beta \geq (q-9)/4$, a contradiction. Thus either $\ind(P)\leq 2$, or $\ind(P)\geq q-\beta-2$.
\end{proof}

Hence, if $\beta\leq q/4-5/2$, we may call the index of a point large or small, according to the two possibilities above.

\begin{prop}\label{lindbl}
Assume $\beta\leq q/4-5/2$ and $q\geq 4$. Then on every tangent to $S$ there is at least one point with large index, and on the (possibly not existing) skew line there are at least two points with large index.
\end{prop}
\begin{proof}
Let $\ell$ be a skew line. A tangent line intersects $\ell$ in a point with index at least three, hence in a point with large index. If there were at most one point with large index on $\ell$, then there would be at most $q$ tangents to $S$, whence the parameter $t$ in Proposition \ref{quadineq} would be at most $q+2$. A point $P$ on $\ell$ with small index has index two, while by inequality \eqref{ineqt} we have $\ind(P)^2-(q-\beta)\ind(P)+q+2=2\beta+6-q\geq0$, in contradiction with $\beta\leq q/4-5/2$ under $q\geq 4$.

Suppose that $\ell$ is tangent to $S$. Suppose that all indices on $\ell$ are at most two. Then there is no skew line to $S$ as the intersection point would have index at least three. Then we have $t\leq 1+q$. If there is a point $P\in [\ell]\setminus{S}$ with index two, \eqref{ineqt} gives $4-2(q-\beta)+q+1=\beta+5-q\geq0$, a contradiction. If all points on $\ell$ have index one, then $t=1$, and \eqref{ineqt} yields $2-q+\beta\geq0$, again a contradiction.
\end{proof}

\begin{theorem}\label{2bl}
Let $S$ be a semi-resolving set in $\PG(2,q)$, $q\geq4$. If $|S|< 2q+q/4-3$, then one can add at most two points to $S$ to obtain a double blocking set.
\end{theorem}
\begin{proof}
Proposition \ref{indices} works, and by Proposition \ref{lindbl} we see that if there is no point with large index, then $S$ is a double blocking set, hence $|S|\geq\tau_2$. Hence we may assume that there are points with large index.

Suppose that there exists a line $\ell$ skew to $S$. Assume that there exist three points with large index on $\ell$. Then the number of tangent lines to $S$ through these points is at least $3(q-\beta-4)>3(3q/4-1)=9q/4-3$. On the other hand, there are at most $|S|<2q+q/4-3$ of them, a contradiction. Thus there are at most two points with large index on $\ell$.

Suppose that there is a point $P$ with large index not on $\ell$. Then the intersection points of the $q-\beta-2>3q/4+1>3$ tangents through $P$ with $\ell$ would have index $\geq 3$, hence a large index, which is not possible. Thus every point with large index is on $\ell$, and hence there are at most two of them.

If there is no skew line to $S$, then the number of points with large index is also at most two, as on three points with large index we would see at least $3(q-\beta-2)-3>9q/4$ tangents to $S$, again a contradiction.

Recall that the points with large index are not in $S$. Add the points with large index to $S$. Then, by Proposition \ref{lindbl}, we obtain a double blocking set.
\end{proof}

\begin{cor}[Theorem \ref{semirestheo}]
Let $S$ be a semi-resolving set in $\PG(2,q)$, $q\geq4$. Then $|S|\geq\min\{2q+q/4-3,\tau_2-2\}$. 
\end{cor}

Recall that $\tau_2(\PG(2,q))=2q+2\sqrt{q}+2$ if $q\geq 9$ is a square (Result \ref{BB}). Thus if $q\geq 121$ is a square, then $\tau_2-2<2q+q/4-3$, hence $\mu_S(\PG(2,q))\geq 2q+2\sqrt{q}$. Proposition \ref{semicon} (ii) shows that equality holds. As $\PG(2,q)$ is isomorphic to its dual, Theorem \ref{semirestheo} holds for lines as well, that is, if we want to resolve the point-set of $\PG(2,q)$ by a set of lines, the same bounds apply. Hence the double of the respective bounds are valid for split resolving sets, so Corollary \ref{semirescor} follows. From the results of \cite{BSS} it follows that in $\PG(2,q)$, $q>256$, any double blocking set of size $2(q+\sqrt{q}+1)$ is the union of two disjoint Baer subplanes. It can be shown that, for $q>9$, if we consider the union $S$ of two disjoint Baer subplanes and take out two points from one of them, then we find a point in the other one on which there are two tangents to $S$. Consequently, assuming the respective bounds on $q$, all semi-resolving sets for $\PG(2,q)$ of size $2q+2\sqrt{q}$ are obtained in Proposition \ref{semicon} (ii).

We remark that for small values of $q$, there are semi-resolving sets smaller than $\tau_2-2$. Three points in general position show $\tau_2(\PG(2,q))=3$. A vertexless triangle (the union of the point-set of three lines in general position without their three intersection points) is easily seen to be a semi-resolving set of size $3q-3$ for $q\geq 3$. If $q\geq4$, we may take out one more (arbitrary) point to obtain a semi-resolving set of size $3q-4$. (In fact, there are no smaller semi-resolving sets than the previous ones for $q=2,3,4$.) On the other hand, $\tau_2(\PG(2,q))=3q$ for $q=2,3,4,5,7,8$ (mentioned in \cite{BB}; this result is due to various authors).

Finally, let us mention an immediate consequence of Theorem \ref{2bl} on the size of a blocking semioval. For more information on semiovals, we refer to \cite{Kiss}.

\begin{defi}
A point-set $S$ in a finite projective plane is a \emph{semioval}, if for all $P\in S$, there is exactly one tangent to $S$ through $P$. A semioval $S$ is a \emph{blocking semioval}, if there are no skew lines to $S$.
\end{defi}

Lower bounds on the size of blocking semiovals are of interest. Up to our knowledge, the following is the best bound known.

\begin{result}[Dover \cite{Dover}]
Let $S$ be a blocking semioval in an arbitrary projective plane of order $q$. If $q\geq7$, then $|S|\geq 2q+2$. If $q\geq3$ and there is a line intersecting $S$ in $q-k$ points, $1\leq k\leq q-1$, then $|S|\geq 3q-2q/(k+2)-k$. 
\end{result}

\begin{cor}\label{blso}
Let $S$ be a blocking semioval in $\PG(2,q)$, $q\geq4$. Then $|S|\geq 9q/4-3$.
\end{cor}
\begin{proof}
By Proposition \ref{semidef1}, $S$ is clearly a semi-resolving set. Suppose to the contrary that $|S|<9q/4-3$. Then by Theorem \ref{2bl}, we find two points, $P$ and $Q$, such that $S\cup\{P,Q\}$ is a double blocking set, that is, $P$ and $Q$ block all the $|S|$ tangents to $S$.
On the other hand, $|S|\geq\tau_2-2>2q+1$ (here we use $q\geq4$ and $\tau_2=3q$ for $q\leq 8$, $\tau_2\geq2q+2\sqrt{q}+2$ for $q\geq9$.) Hence $S$ has more than $2q+1$ tangents. However, $P$ and $Q$ can block at most $2q+1$ of them, a contradiction.
\end{proof}

Note that Dover's result is better than Corollary \ref{blso} if there is a line intersecting the blocking semioval in more than $q/4$ points (roughly).


\subsection*{Acknowledgements}
We are grateful for the advices and support of P\'eter Sziklai and Tam\'as Sz\H{o}nyi. We are also thankful to Bence Csajb\'ok for pointing out the connection of semi-resolving sets and blocking semiovals.

\newpage

 \textheight 25 cm
 \topmargin -1cm


\begin{figure}[!h]
\hspace{-20mm}\vspace{5mm}
\scalebox{0.95}{
\begingroup\makeatletter\ifx\SetFigFontNFSS\undefined%
\gdef\SetFigFontNFSS#1#2#3#4#5{%
  \reset@font\fontsize{#1}{#2pt}%
  \fontfamily{#3}\fontseries{#4}\fontshape{#5}%
  \selectfont}%
\fi\endgroup%

\renewcommand{\mb}{\makebox(0,0)}
\renewcommand{\cd}{3}
\renewcommand{\cdd}{3.7}
\renewcommand{\dll}{1.75}
\renewcommand{\fontdef}{\SetFigFontNFSS{10}{10}{\rmdefault}{\mddefault}{\updefault}}
\renewcommand{\fontsmall}{\SetFigFontNFSS{8}{8}{\rmdefault}{\mddefault}{\updefault}}

\setlength{\unitlength}{0.63mm}
\renewcommand{\dashlinestretch}{40}

\fontdef
\thinlines


}
\caption{\label{consfig} The 32 types of resolving sets of size $4q-4$ with $|\PS|\leq|\LS|$.}
\end{figure}

\newpage
\centerline{\Large \textbf{Clarifications and minor corrections}}



\bigskip

In Section 3 (Constructions), the following additions have been made.

\begin{itemize}
\item{{\bf 4.} In this case, as $R'\in\Pe_S$, actually $U=R'$ may also occur. Thus
\begin{enumerate}
\item in (C10), we may add $R'Q$ as well to solve P2' (p12, l22: `or $R'Q$');
\item in (C12), we may add the line connecting $R'$ and $\ell_0 \cap RQ$ to solve P2' (p12, l23: `or $R'$').
\end{enumerate}}
\item{
\noindent{\bf 9. b)} To be coherent, we have noted that $Z=\ell_0\cap R'Q$ is forbidden in (C29) as well, and $Z=Q$ is possible in (C30) (p13, l-7: `(in both cases)'; p13, l-6: `or $Q\in R'Q$').
}
\end{itemize}

Parts (C10), (C12), (C15), (C29), (C30), (C31), (C32) and the legend of Figure 3 have been updated accordingly. These changes mostly clarify special cases; e.g., in (C29), $Z\in \ell_0$ may be $P$ or $\ell_0\cap R'T$ as well, and thus, not to suggest the contrary, these points are not depicted as outer points anymore but as possibly added points. So a possibly added point $U$ is in $\PS$ if and only if we choose $Z=U$. Another type of correction regards the added lines; e.g., in (C10), the added line may intersect $e$ in any point of $\PS$ including the added point $R'$; hence we have not depicted the intersection as a point of $S^*$. Furthermore, some subtle aesthetical changes have been made without mention.

\end{document}